\documentclass[a4paper]{amsart}
\usepackage{amssymb}
\usepackage{xypic}
\usepackage{graphicx}
\usepackage[utf8]{inputenc}

\newtheorem{theorem}{Theorem}[section]
\newtheorem{lemma}[theorem]{Lemma}
\newtheorem{cor}[theorem]{Corollary}

\theoremstyle{definition}
\newtheorem{definition}[theorem]{Definition}

\theoremstyle{remark}
\newtheorem{remark}[theorem]{Remark}

\numberwithin{equation}{section}

\begin{document}

\title{Smooth classification of locally standard $T^k$-manifolds}


\author{Michael Wiemeler}
\address{Mathematisches Institut\\ WWU M\"unster\\ Einsteinstra\ss{}e 62\\ 48149 M\"unster\\Germany}
\curraddr{}
\email{wiemelerm@uni-muenster.de}




\date{\today}

\dedicatory{}

\begin{abstract}
  We study locally standard \(T^k\)-manifolds \(M\).
  In particular, we study the case where there is a continuous section to the orbit map \(\pi:M\rightarrow M/T\).
  We give a classification of \(T^k\)-manifolds satisfying these conditions up to equivariant diffeomorphism.
\end{abstract}

\maketitle

\section{Introduction}

A smooth action of a torus \(T^k\) on a manifold \(M\) (without boundary, not necessary compact) is called locally standard if each \(T^k\)-orbit in \(M\) has an invariant neighborhood which is equivariantly diffeomorphic to some \(V=\mathbb{C}^n\times T^{k-n}\times \mathbb{R}^m\), where \(n\) and \(m\) may depend on the orbit.
Here the action on \(V\) can be described as follows: There is a splitting \(T^k=T^n\times T^{k-n}\) such that \(T^n\) acts linearly and effectively on \(\mathbb{C}^n\), \(T^{k-n}\) acts by left multiplication on itself and \(T^k\) acts trivially on \(\mathbb{R}^m\).
The action on \(V\) is the product of these actions.

If \(\dim M=2k\) and there is a \(T^k\)-fixed point in \(M\) then \(M\) is a torus manifold and it is locally standard in the usual sense if and only if it is locally standard in the sense of the above definition.
The orbit space of a locally standard \(T^k\)-manifold is naturally a manifold with corners.
If \(M\) is a torus manifold and \(M/T\) is homeomorphic to a simple convex polytope, then \(M\) is said to be quasitoric.
The study of quasitoric manifolds began with \cite{DJ91}.
In \cite{wiem13}, they have been classified up to equivariant diffeomorphism in terms of combinatorial data.
Later in \cite{wiem15} this classification has been extended to locally standard torus manifolds with shellable orbit spaces all of whose closed faces are diffeomorphic to standard discs after smoothing corners.

Later on in \cite{sarkar19:_equiv_cohom_rigid_topol_contac_toric_manif} topological contact toric manifolds have been introduced and classified up to equivariant homeomorphism.
These are also locally standard \(T^k\)-manifolds in our sense.

In this note we prove a classification of general locally standard \(T^k\)-manifolds up to equivariant diffeomorphism.
Our main result is  as follows.

\begin{theorem}
  \label{sec:introduction}
  Let \(M_i\), \(i=1,2\), be two locally standard \(T^k\)-manifolds with orbit spaces \(P_i\) and characteristic functions \(\lambda_i\).
  Assume that there are sections \(s_i:P_i\rightarrow M_i\) to the orbit maps.

  If there is a diffeomorphism
  \[\Phi:P_1\rightarrow P_2\]
  such that \[\lambda_1=\lambda_2\circ \Phi,\] then there is an equivariant diffeomorphism \(\Psi:M_1\rightarrow M_2\).
\end{theorem}

In the above theorem the characteristic function \(\lambda\) is a map from the orbit space of a locally standard \(T^k\)-manifold \(M\) to the set of subtori of \(T^k\) which assigns to each orbit its isotropy group.
Note that \(\lambda\) is constant on open faces of \(M/T\).
Therefore we might also think of \(\lambda\) as defined on the set of faces of \(M/T\).

Note that there is always a section to the orbit map \(M\rightarrow M/T\) if \(H^2(M/T;\mathbb{Z})=0\).

By combining the above theorem with Theorem 4.2 of \cite{davi14} we also get:

\begin{theorem}
  Let \(M_i\), \(i=1,2\), be two locally standard \(T^k\)-manifolds with orbit spaces \(P_i\) and characteristic functions \(\lambda_i\).
Assume that all closed faces of the \(P_i\) are contractible.

  If there is a isomorphism of posets
  \[\Phi:\mathcal{P}(P_1)\rightarrow \mathcal{P}(P_2)\]
  such that \[\lambda_1=\lambda_2\circ \Phi,\]
  and \(\Phi(F)\) is diffeomorphic to \(F\) after smoothing corners for all four-dimensional faces \(F\) of \(P_1\) then there is an equivariant diffeomorphism \(\Psi:M_1\rightarrow M_2\).
\end{theorem}

Here \(\mathcal{P}(P)\) denotes the face poset of \(P\), i.e. the set of faces of \(P\) partially ordered by inclusion.

Theorem~\ref{sec:introduction} is probably known to experts.
It follows for example from results of \cite{MR1116630}. But as pointed out in \cite{karshon20:_equiv} there is a gap in the arguments of {\cite{MR1116630}}.
However, we could not find another proof in the literature. Moreover, because the arguments presented here are much simpler than those given for example in \cite{karshon20:_equiv} or \cite{wiem13} we think they are worth to be written down.

This note is structured as follows.
In the next section we describe local properties of locally standard torus actions. Then we show the existence of so-called regular cross-sections to orbit maps if there is a section to the orbit map.
This is used in the last section to proof Theorem~\ref{sec:introduction}.

I would like to thank the Fields Institute in Toronto, Canada, for hospitality where the work for this paper was carried out.
I also want to thank the participants of the Thematic Program on Toric Topology and Polyhedreal Products 2020 at the Fields Institute for discussions on the subject of this paper.
In particular, I want to thank Matthias Franz and Mikiya Masuda.

This research was supported by Deutsche Forschungsgemeinschaft through CRC 1442 Geometry: Deformations and Rigidity and through the Cluster of Excellence Mathematics Münster.

\section{Local coordinates}

In this section we describe local properties of orbit spaces of locally standard \(T^k\)-actions and introduce a smooth structure on these orbit spaces.

We start with a lemma.

\begin{lemma}
  Let \(f:\mathbb{R}_{\geq 0}^n\times \mathbb{R}^m\rightarrow \mathbb{R}\) be a map. Then \(f\) is smooth if and only if \[F:\mathbb{R}^n\times \mathbb{R}^m\rightarrow \mathbb{R} \quad\quad (x_1,\dots,x_n,y_1,\dots,y_m)\mapsto f(x_1^2,\dots,x_n^2,y_1,\dots,y_m)\] is smooth.
\end{lemma}
\begin{proof}
  Note that when \(f\) is smooth it immediately follows that \(F\) is smooth.
  Therefore we only have to show the other implication.
  
  For \(n=1\) this was done in Section 5 of Chapter VI in \cite{bred72}.
  To prove the general case we proceed by induction on \(n\).
 Assume that the lemma holds for some fixed \(n\in\mathbb{N}\). We have to show that this implies that it holds for \(n+1\).
  This is similar to the arguments in the proof in Bredon's book. Therefore we only indicate the idea.

  Let \(f:\mathbb{R}_{\geq 0}^{n+1}\times \mathbb{R}^m\rightarrow \mathbb{R}\) such that the corresponding \(F\) is smooth. Then it follows from the induction hypothesis that \(f\) is smooth away from \(\{x_1=0\}\).
  Since \(F\) is an even function in all \(x_i\)-coordinates its Taylor-expansion at \(x_1=0\) up to order \(2r+1\) is of the form
  \[\begin{split}h_0(x_2,\dots,x_{n+1},y_1,\dots,y_m)+h_1(x_2,\dots,x_{n+1},y_1,\dots,y_m)x_1^2+\dots\\ + h_r(x_2,\dots,x_{n+1},y_1,\dots,y_m)x_1^{2r}\end{split}\]
  Note that the \(h_i\) are smooth functions which are even in all \(x_i\)-variables.
  Hence by the induction hypothesis, they are of the form
  \[h_i(x_2,\dots,x_{n+1},y_1,\dots,y_m)=\bar{h}_i(x_2^2,\dots,x_{n+1}^2,y_1,\dots,y_m),\]
  with some smooth map \(\bar{h}_i:\mathbb{R}_{\geq 0}^n\times \mathbb{R}^m\rightarrow \mathbb{R}\).
  Hence, it follows from the above form of the Taylor expansion of \(F\) that \[f(x_1,\dots,x_{n+1},y_1,\dots,y_m)=F(\sqrt{x_1},\dots,\sqrt{x_{n+1}},y_1,\dots,y_m)\]
  is of class \(C^r\) at \(\{x_1=0\}\). Since this holds for all \(r\), it follows that \(f\) is smooth.
\end{proof}

Now let \(T^k\) act on \(M=\mathbb{C}^n\times T^{k-n}\times \mathbb{R}^m\) in the standard way.
Then the map
\[\pi:M\rightarrow \mathbb{R}^n_{\geq 0}\times \mathbb{R}^m\quad \quad (z_1,\dots,z_n,t,y)\mapsto (|z_1|^2,\dots,|z_n|^2,y)\]
induces an identification of \(M/T\) with \(\mathbb{R}_{\geq 0}^n\times \mathbb{R}^m\).
From the above lemma we immediately get the following corollary.

\begin{cor}
  A map \(f:\mathbb{R}_{\geq 0}^n\times \mathbb{R}^m\rightarrow \mathbb{R}\) is smooth if and only if \(f\circ \pi\) is smooth.
  In particular every smooth torus invariant map \(M \rightarrow \mathbb{R}\) induces a smooth map on the orbit space.
\end{cor}

\begin{remark}
  The above lemma and corollary are special cases of a more general result of G.W. Schwarz (see \cite{schwarz75}).
\end{remark}

Applying this corollary to \(\pi\circ f\) where \(f\) is a transition function of invariant smooth charts on a locally standard \(T^k\)-manifold implies the following.

\begin{cor}
  The orbit space of a locally standard \(T^k\)-manifold is a smooth manifold with corners.
\end{cor}

In the following we have to deal with \(T^k\)-manifolds of a particular type which we call conical.

\begin{definition}
  Let \(M\) be a locally standard \(T^k\)-manifold and \(U\subset M/T\) an open subset.
  We say that \(U\) is conical if it is diffeomorphic to \(\mathbb{R}_{\geq 0}^{n}\times N\), where \(N\) is some manifold without boundary.
  We say that \(M\) is conical if \(M/T\) is conical.
\end{definition}

For conical \(T^k\)-manifolds we can show the following theorem.

\begin{theorem}
  \begin{enumerate}
  \item   Every conical \(T^{k}\) manifold \(M\) over \(\mathbb{R}_{\geq 0}^{n}\times N\)
  is equivariantly diffeomorphic to the normal bundle of \(N_0=\pi^{-1}(\{0\}\times N)\) in \(M\).
  \item If there is a section to the orbit map \(\pi:M\rightarrow M/T\), then this normal bundle is trivial, i.e. (weakly) equivariantly diffeomorphic to  \(\mathbb{C}^{n}\times T^{k-n}\times N\).
  \end{enumerate}
\end{theorem}
\begin{proof}
  \begin{enumerate}
  \item   Note that \(N_0\) is a invariant submanifold of \(M\).
    Therefore there is a tubular neighborhood of \(N_0\) which is equivariantly diffeomorphic to the normal bundle of \(N_0\).
    Let \(d:\mathbb{R}_{\geq 0}^{n}\rightarrow \mathbb{R}\) be the Euclidean distance from \(0\) and \(d_0=d\circ p\) where \(p:\mathbb{R}_{\geq 0}^{n}\times N\rightarrow \mathbb{R}_{\geq 0}^{n}\) is the projection.
  Then \(d_0\) is proper and smooth away from \(\{0\}\times N\).
  Moreover, all points in \(M/T-(\{0\}\times N)\) are regular.
  Therefore the claim follows from the Morse lemma applied to \(d_0\circ \pi\).

\item Since there is a section to the orbit map, \(N_0\) is diffeomorphic to \(T^{k-n}\times N\).
  Since the normal bundle of \(N_0\) splits as a sum of complex line bundles it suffices to consider the case \(n=1\). In that case a section to the normal bundle \(E\rightarrow N_0\) is given by \(s\circ t_\epsilon\circ \pi\), where \(t_{\epsilon}:\{0\}\times N_0\rightarrow \{\epsilon\}\times N_0\) is the obvious diffeomorphism for some small \(\epsilon\) and \(s\) is the section to the orbit map.

  Therefore the second claim follows.
  \end{enumerate}
\end{proof}

We also need to know the equivariant self-diffeomorphisms of conical \(T^k\)-manifolds.

\begin{theorem}
\label{sec:sect-local-coord}
  Let \(M=\mathbb{C}^{n}\times T^{k-n}\times N\) then the group \(\mathcal{D}(M)\)  of equivariant diffeomorphisms of \(M\) is isomorphic to
  \[C^\infty(M/T,T^k)\rtimes \mathrm{Diff}(M/T),\]
where \(\mathrm{Diff}(M/T)\) is the group of all diffeomorphisms of \(M/T\) which map each face of \(M/T\) to itself and \(C^\infty(M/T,T^k)\) denotes the group of smooth maps \(M/T\rightarrow T^k\).

Indeed, every equivariant diffeomorphism \(\Psi\)  of \(M\) is of the form

\[\begin{split}\Psi(z)&=\Psi(z_1,\dots,z_{n},t,x)=\\ &\left(z_1f_1(\pi(z))\sqrt{\frac{\Phi_1(\pi(z))}{|z_1|^2}},\dots,z_nf_n(\pi(z))\sqrt{\frac{\Phi_n(\pi(z))}{|z_n|^2}},t h(\pi(z)),\Phi_N(\pi(z))\right)\end{split}\]

where \(\Phi:M/T\rightarrow M/T\) is a diffeomorphism, \(\Phi_i\) denotes the \(i\)-th component of \(\Phi\) and \(f_i:M/T\rightarrow S^1\) and \(h:M/T\rightarrow T^{k-n}\) are smooth.
    Moreover, for every choice of such maps \(\Phi, f_i, h\), one gets a diffeomorphism \(\Psi\) as above.
\end{theorem}
\begin{proof}
  This is basically a local version of Lemma 2.3 in \cite{wiem18}. But note that in that paper the group of weakly equivariant diffeomorphisms was considered, whereas here the group of (strongly) equivariant diffeomorphisms is used.
  In the paper \cite{wiem18} we considered the normalizer of the torus in the group of all diffeomorphism.
  Here we consider the centralizer of the torus.

  For the sake of completeness we give a complete proof here.

  We have two homomorphisms of topological groups:
  \begin{align*}
    l_1:\mathcal{D}(M)&\rightarrow \mathrm{Diff}(M/T)& f&\mapsto([x]\mapsto [f(x)])\\
    l_2: C^\infty(M/T^k, T^k)&\rightarrow \ker l_1& f&\mapsto(x\mapsto x \cdot f(\pi(x))).
  \end{align*}

  To prove the theorem if suffices to show that
  \begin{enumerate}
  \item There is a section to \(l_1\).
  \item \(l_2\) is surjective.
  \end{enumerate}

  {\emph{ Ad (1).}} Let \(\Phi\) be a diffeomorphism of \(M/T\).
  Then
  \[\Psi(z_1,\dots,z_n,t,y)=\left(\sqrt{\Phi_1(\pi(z,t,y))\frac{1}{|z_1|^2}} z_1,\dots,\sqrt{\Phi_n(\pi(z,t,y))\frac{1}{|z_n|^2}} z_n,t,\Phi_N(y)\right)\]
  is a \(T^k\)-equivariant homoemomorphism of \(M\) covering \(\Phi\).
  We claim that \(\Psi\) is of class \(C^\infty\).
  Since our argument is also valid for the inverses of \(\Phi\) and \(\Psi\), this shows that \(\Psi\) is a diffeomorphism.
  To show that \(\Psi\) is \(C^\infty\) we can deal with each component sperately.

  It is clear that the last two components are smooth therefore we only have to deal with the components of the form
  \[\sqrt{\Phi_i(\pi(z,t,x))\frac{1}{|z_i|^2}} z_i.\]

  Since \(\Phi_i(\pi(z,t,x))=0\) if and only if \(z_i=0\), we only have to check differentiability of this component at points with \(z_i=0\).
  Note that since \(\Phi\) is a diffeomorphism of \(M/T\) which preserves the face structure of \(M/T\) we have
  \[\frac{\partial\Phi_i}{\partial |z_j|^2}|_{z_i=0}=0 \quad\quad \text{ for } j\neq i\]
    and 
    \[\frac{\partial\Phi_i}{\partial |z_i|^2}|_{z_i=0}> 0.\]
    
Therefore smoothness of the above component follows from Lemma \ref{technical_lemma} below.
 
  \emph{Ad (2).}
 We show that the kernel of the natural map \(l_1:\mathcal{D}(M)\rightarrow \mathrm{Diff}(M/T)\) is isomorphic to \(C^{\infty}(M/T,T^k)\).
  Since \(T^k\) is abelian, there is a natural map \(l_2\) from \(C^{\infty}(M/T,T)/T\) to the kernel of \(l_1\). Namely \(f\mapsto (x\mapsto x\cdot f(\pi(x))\).

We show that this map is a homeomorphism.
To do so, let \(F\in \ker l_1\).
Then \(F\) leaves all \(T\)-invariant subsets of \(M\) invariant.
Since \(M\) is conical \(F\) has the following form:
\begin{align*}
  F(z_1,\dots,z_n,t,y)=(z_1&f_1(z_1,\dots,z_n,y),\dots,z_nf_n(z_1,\dots,z_n,y),\\ &th(z,y),F_N(z_1,\dots,z_n,y)),
\end{align*}
where \((z_1,\dots,z_n)\in \mathbb{C}^n\), \(t \in T^{k-n}\),\(y\in N\) and \(f_j(z_1,\dots,z_n,y)\in S^1\), \(h(z,y)\in T^{k-n}\), for \(j=1,\dots,n\) depends only on \((|z_1|^2,\dots,|z_n|^2)\) and \(y\).

We have to show that \(f_j\) is smooth for all \(j\).

Smoothness in points with \(z_j\neq 0\) follows from the smoothness of \(F\).
We show that \(f_j\) is also smooth in points with \(z_j=0\).

Since \(F\) is smooth, by the fundamental theorem of calculus (applied to the derivative of the function \(t\mapsto F_j(z_1,\dots,z_{j-1},z_jt,z_{j+1},\dots,z_n,y)\)), we have for \((z_1,\dots,z_n)\in \mathbb{C}^n\),
\begin{align*}
  z_jf_j(z_1,\dots,z_n,y)&=F_j(z_1,\dots,z_n,y)=\\ &\int_0^1(D_{z_j}F_j(z_1,\dots,z_{j-1},z_jt,z_{j+1},\dots,z_n,y))(z_j) \; dt,
\end{align*}
where
\begin{equation*}
  (D_{z_j}F_j(z_1,\dots,z_n,y))(z)=\left(\frac{\partial F_j}{\partial x_j}(z_1,\dots,z_n,y),\frac{\partial F_j}{\partial y_j}(z_1,\dots,z_n,y)\right)(v,w)^t
\end{equation*}
with \(z_l=x_l+iy_l\) for \(l=1,\dots,n\) and \(z=v+iw\), \(x_l,v,y_l,w\in \mathbb{R}\) is the derivative of \(F_j\) in the point \((z_1,\dots,z_n,y)\) in direction \(z\) in the \(j\)-th coordinate.

Since \(F\) is \(T\)-equivariant,
we have \[zF_j(z_1,\dots,z_n)=F_j(z_1,\dots,z_{j-1},zz_j,z_{j+1}\dots,z_n)\] for \(z\in S^1\subset \mathbb{C}\).
Hence it follows that
\[zD_{z_j}F_j(z_1,\dots,z_n)(z')=D_{z_j}F_j(z_1,\dots,z_{j-1},zz_j,z_{j+1},\dots,z_n)(zz')\]
for \(z'\in \mathbb{C}\).

Therefore it follows that
\begin{align*}
  z_jf_j(z_1,\dots,z_n,y)&=\int_0^1 (D_{z_j}F_j(z_1,\dots,z_{j-1},z_jt,z_{j+1},\dots,z_n,y))(z_j) \; dt\\
&=\int_0^1z_j (D_{z_j}F_j(z_1,\dots,z_{j-1},\frac{|z_j|}{z_j} z_jt,z_{j+1},\dots,z_n,y))(1) \; dt\\
&=z_j\int_0^1 (D_{z_j}F_k(z_1,\dots,z_{j-1},|z_j|t,z_{j+1},\dots,z_n,y))(1) \; dt.
\end{align*}

Since \(F\) is \(T\)-equivariant, it follows that \[t\mapsto
(D_{z_j}F_k(z_1,\dots,z_{j-1},t,z_{j+1},\dots,z_n,y))(1),\] \(t\in \mathbb{R}\), is an even function.
Therefore the integrand in the last integral depends smoothly on \((z_1,\dots,z_n,y)\) and \(f_j\) is smooth everywhere.
Because \(f_j\) is \(T\)-invariant, it induces a smooth map on the orbit space, whose derivatives depend continuously on the derivatives of \(F\).

Hence the theorem is proved.

\end{proof}

\begin{lemma}
  \label{technical_lemma}
  Let \(f:\mathbb{R}_{\geq 0}\times \mathbb{R}^{n-1}\rightarrow \mathbb{R}\), \((x,y)\mapsto f(x,y)\) be a smooth function such that:
  \begin{enumerate}
  \item \(f(0,\cdot)=0\),
  \item \(\frac{\partial f}{\partial x}(0,\cdot)>0\), and
  \item \(f(x,y)>0\) if \(x>0\).
  \end{enumerate}
  Then \(g:\mathbb{R}_{\geq 0}\times \mathbb{R}^{n-1}\rightarrow \mathbb{R}\),
  \[g(x,y)=\begin{cases}\frac{f(x,y)}{x}&\text{if } x>0\\
      \frac{\partial f}{\partial x}(0,y)& \text{if }x=0
    \end{cases}\]
  is a smooth positive function.
  In particular, \(\sqrt{g}\) is a smooth positive function.
\end{lemma}
\begin{proof}
  We only have to prove smoothness of \(g\) at \(\{x=0\}\).
  To do so note that the Taylor expantion of \(f\) at \(\{x=0\}\) up to order \(r\) with respect to the first coordinate is of the form
  \[p_r(x,y)=\frac{\partial f}{\partial x}(0,y)x + \sum_{i=2}^r h_i(y)x^i\]
  with smooth functions \(h_i:\mathbb{R}^{n-1}\rightarrow \mathbb{R}\).
  Dividing this expression by \(x\) shows that \(g\) is of class \(C^{r-1}\).
  Since this holds for all \(r\) smoothness of \(g\) follows.
\end{proof}

\section{Existence of regular sections}
\label{sec:exist-regul-sect-1}

One technical problem in the classification of smooth locally standard \(T^k\)-mani\-folds is that there is never a section to the orbit map of such a manifold which is smooth at the boundary of \(M/T\).
However, in this section we show that if there is a continuous cross-section to the orbit map then there is also a section with enough regularity to carry out the proof of the smooth classification.
We call these sections regular sections.
We start with the setup for their definition.

 Let \(M\) be a locally standard \(T^k\)-manifold and \(s:M/T\rightarrow M\) a section to the orbit map.
  Let \(x\in M/T\). Let \(U\subset M/T\) be a conical neighborhood of \(x\).
  Then by the previous section \(\pi^{-1}(U)\) is diffeomorphic to \(\mathbb{C}^n\times T^{k-n}\times N\).
  In these coordinates \(s|_U\) is of the form
  \(s(x,y)=f(x,y)s_0(x,y)\), where \(s_0(x,y)=s_0(x_1,\dots,x_n,y)=(\sqrt{x_1},\dots,\sqrt{x_n},1,y)\) for  \(x=(x_1,\dots,x_n)\in \mathbb{R}_{\geq 0}^n\), \(y\in N\), \(f:U\rightarrow T^k\) some map which might be non-continuous at the boundary.
  Note that in general \(f\) is not unique.
  
\begin{definition}
  We say that \(s\) is regular at \(x\) if \(f\) can be choosen so that it is smooth in an open neighborhood \(V\subset U\) of \(x\).
  We say that \(s\) is regular if it is regular at every \(x\in M/T\).
\end{definition}

We need to know that the above definition does not depend on the choice of coordinates around \(x\).
This is done by the following lemma.

\begin{lemma}
  If \(s|_U\) is of the form in the definition for some choice of coordinates, then it is also of this form for all other coordinates.
\end{lemma}

\begin{proof}
  Let \(U_1,U_2\) be to conical neighborhoods of \(x\). We can assume that they are both equal as subsets of \(M/T\). Moreover, we can assume that they are both diffeomorphic to the same standard model \(U_{std}=\mathbb{R}_{\geq 0}^n\times N\) with \(\pi^{-1}(U_{std})=\mathbb{C}^n\times T^{k-n}\times N\). Let \(\bar\Phi_i:U_{std}\rightarrow U_i\) be diffeomorphisms and \(\bar\Phi=\bar\Phi_1^{-1}\circ \bar\Phi_2\) the transition function. Let \({\Phi}\) be a equivariant diffeomorphism \(\pi^{-1}(U_{std})\rightarrow \pi^{-1}(U_{std})\) covering \(\bar\Phi\). Write \(s_1(x)=\Phi_1^{-1}\circ s\circ\bar\Phi_1(x)=f(x)\cdot s_0(x)\).
  Then
  \[s_2(x)=\Phi_2^{-1}\circ s \circ\bar\Phi_2(x)=\Phi^{-1}\circ s_1\circ\bar{\Phi}(x)=f(x)\Phi^{-1}\circ s_0\circ \bar{\Phi}(x)=g(x)f(x)s_0(x)\]
  with \(g:U_{std}\rightarrow T^k\) smooth as in Theorem~\ref{sec:sect-local-coord}.
  Therefore if \(s_1\) is regular at \(x\) it follows that \(s_2\) is also regular at \(x\).
\end{proof}

After this preparation we can show the existence of regular sections.

\begin{theorem}
  \label{sec:exist-regul-sect}
  Let \(M\) be a locally standard \(T^k\)-manifold such that there is a section to
  \(\pi:M\rightarrow M/T\).
  Then there is a regular section.
\end{theorem}
\begin{proof}
  There is a partial ordering of the open faces of \(M/T\) as follows:
  \begin{equation*}
    F\geq F'\quad\quad\Leftrightarrow\quad\quad F'\subset \overline{F},
  \end{equation*}
  where \(\overline{F}\) denotes the closure of \(F\).
  Order the faces \(F_0,F_1,\dots,F_m\) of \(M/T\) in such a way that
  \begin{equation*}
   i\leq j \quad\quad\text{if}\quad\quad F_i\geq F_j.
 \end{equation*}
 For \(j=0,\dots, m\) let \(M_j=\pi^{-1}(\bigcup_{i=0}^j F_i)\). Then \(M_0\) is the union of principal orbits and \(M_m=M\).
 Note that there is a smooth section to \(\pi|_{M_0}:M_0\rightarrow M_0/T\).

 Therefore it suffices to show the following claim:

 \textbf{Claim:} If there is a regular section to \(\pi|_{M_j}\), then there is also a regular section to \(\pi|_{M_{j+1}}\).

   To see this let \(E_{j+1}\) be a tubular neighborhood of \(F_{j+1}\subset M_{j+1}/T\).
   Note \(E_{j+1}\) is conical.

   Therefore the problem of extending a regular section on \(M_{j}/T\) to a regular section on \(M_{j+1}/T\) is the same as extending a smooth map \(f:E_{j+1}-F_{j+1}\rightarrow T^k\) to all of \(E_{j+1}\).
   Note that
   \(E_{j+1}\) is diffeomorphic to \(F_{j+1}\times \mathbb{R}^n_{\geq 0}\), via a diffeomorphism which maps \(F_{j+1}\) to \(F_{j+1}\times \{0\}\).
  Hence \(E_{j+1}-F_{j+1}\) is homotopy equivalent to \(E_{j+1}\). Therefore we can perturb \(f\) so that it can be extended to all of \(E_{j+1}\).
   Hence the claim and the theorem are proven.
\end{proof}

\section{The proof of the main result}

After establishing the existence of regular sections we can now prove our main result.

\begin{theorem}
  \label{sec:proof-main-result}
  Let \(M_i\), \(i=1,2\), be two locally standard \(T^k\)-manifolds with orbit spaces \(P_i\) and characteristic functions \(\lambda_i\).
  Assume that there are regular sections \(s_i:P_i\rightarrow M_i\).

  If there is a diffeomorphism
  \[\Phi:P_1\rightarrow P_2\]
  such that \[\lambda_1=\lambda_2\circ \Phi,\] then there is a unique equivariant diffeomorphism \(\Psi:M_1\rightarrow M_2\), such that
  \begin{equation}
    \label{eq:1}
  \Psi\circ s_1=s_2\circ \Phi.
  \end{equation}
\end{theorem}
\begin{proof}
  The existence of a unique equivariant homeomorphism \(\Psi\) such that (\ref{eq:1}) holds, follows as in the proof of Proposition 1.8 of \cite{DJ91}.
  Therefore it suffices to show that \(\Psi\) is smooth.
  Smoothness is a local property. Therefore it suffices to show the claim for the local models of \(M_i\), that is for conical \(M_i\); say \(M_i\cong \mathbb{C}^n\times T^{k-n}\times \mathbb{R}^m\) and \(P_i\cong \mathbb{R}_{\geq 0}^n\times \mathbb{R}^m\).

  Since the sections \(s_i\), \(i=1,2\), are regular there exist smooth functions \(f_i:P_i\rightarrow T^n\) such that
  \[s_i(x)=f_i(x)\cdot s_0(x),\]
  where \(s_0\) is as in Section \ref{sec:exist-regul-sect-1}.
  So, from (\ref{eq:1}), we have
  \[\Psi(z)=\Psi(z_1,\dots,z_n,t,y)=\left(f_2(\Phi(\pi_1(z))) \cdot f_1(\pi_1(z))^{-1}\right)\cdot \left(\sqrt{\frac{\Phi_i( \pi_1(z))}{|z_i|^2}}z_i,t,\Phi_N(y)\right).\]
    Note that the \(f_i\) and \(\pi_1\) are smooth and \(\Phi\) is a diffeomorphism.
  Hence \(\Psi\) is a diffeomorphism by Theorem~\ref{sec:sect-local-coord}.
\end{proof}

Now Theorem~\ref{sec:introduction} follows from Theorems~\ref{sec:exist-regul-sect} and \ref{sec:proof-main-result}.

\bibliography{smooth}{}
\bibliographystyle{alpha}

\end{document}